\documentclass[10pt]{amsart}
\usepackage{graphicx}
\usepackage[english,francais]{babel}
\usepackage{color}
\usepackage[ansinew]{inputenc}
\usepackage[active]{srcltx}
\usepackage[all]{xy}
\usepackage{amssymb}
\usepackage{amsmath}

\newtheorem{thm}{Theorem}[section]
\newtheorem{cor}[thm]{Corollary}

\newtheorem{lem}[thm]{Lemma}
\newtheorem{prop}[thm]{Proposition}
\theoremstyle{definition}
\newtheorem*{dem}{Proof}

\newtheorem{defi}[thm]{Definition}
\newtheorem{remark}[thm]{Remark}

\theoremstyle{remark}
\newtheorem{rem}[thm]{Remark}

\numberwithin{equation}{section}

\newcommand{\IR}{\mathbb R}

\newcommand{\IN}{\mathbb N}

\newcommand{\field}[1]{\mathbb{#1}}
\title{Jet schemes of toric surfaces \\
\vspace{0.5cm}
Espaces de jets des surfaces toriques
}

\author{ Hussein MOURTADA }

\address{Laboratoire de Mathématiques de Versailles, UMR CNRS 8100 \\ bât. de Fermat, 45  avenue des états-unis, F-78035 Versailles cedex
France}

\email{mourtada@math.uvsq.fr}

\begin{document}
\maketitle

\selectlanguage{english} 

\begin{abstract}
For $m\in \IN, m\geq 1,$ we determine the irreducible components of the $m-th$ jet scheme of a toric surface $S.$ For $m$ big enough, we connect the number of a class of these irreducible components to the number of exceptional divisors on the minimal resolution of $S.$

\vskip 0.5\baselineskip

\selectlanguage{francais}
\noindent{RÉSUMÉ.}
Pour $m \in \IN, m\geq 1,$ on détermine les composantes irréductibles des $m-$espaces des jets d'une surface torique $S$. Pour $m$ assez grand, on relie le nombre d'une classe de ces composantes au nombre de diviseur exceptionnel sur la résolution minimale de $S$. 

\end{abstract}

\selectlanguage{english}

\section{Introduction}
Nash has related the space of arcs centered in the singular locus of a variety 
to its resolution of singularities  in $1968$ (see \cite{N}). Since the late nineties till nowadays, these schemes and their finite dimensional approximations -Jet schemes- have generated much interest because of their appearence in motivic integration(\cite{K},\cite {DL}) and their use in birationnal geometry  \cite{EM}.\\
Despite this appearence presence of these jet schemes in numerous articles and in many interesting questions, few is known about their geometry for specific class of singularities, except for the following three classes: monomial ideals \cite{GS},
determinantal varieties \cite{D}, plane branches \cite{Mo1}. \\
While arcs on toric varieties have been intensively studied (\cite{Mum},\cite{L},\cite{B-GS},\cite{IK}), jet schemes of such varieties are still unknown. The subject of this note is the study of the jet schemes of toric surfaces. Beside being the simplest toric singularities,  this class of singularities is interesting from two points of view: 
on one hand these surfaces are examples of varieties having rational singularities, but which are not necessary local complete intersection, therefore we can not characterize their rationality by \cite{Mus} via their jet schemes ; on the other hand, despite that these singularities are not complete intersections and therfore we do not have a definition of non-degeneration with respect to their Newton polygon in the sense of Kouchnirenko \cite{Ko},  they heuristically are  non-degenerate because they are desingularized with one toric morphism, so from  a jet-scheme theoretical point of view, they should not have vanishing components \cite{Mo1} (i.e. projective systems of irreducible components whose limit in the arc space are included in the arc space of the singular locus) ; this follows from remark \ref{nv}.
For $m\in \IN, m\geq 1,$ we determine the irreducible components of the $m-th$ jet scheme $S_m$ of a toric surface $S.$ We give formulas for their number and their dimensions in terms of $m,$ and invariants of the cone that defines $S.$ For a given $m,$ we classify these irreducible components by an integer invariant that we call index of speciality. We prove that for $m$ big enough, the components with index of speciality $1,$ is in $1-1$ correspondance with the exceptional divisors that appear on the minimal resolution of $S$.  This is to compare with a result that we have obtained in \cite{Mo2} for rational double point singularities.\\

AKNOWLEDGEMENTS\\

I would like to thank Monique Lejeune-Jalabert whose influence radically improved this note. I would like also to thank Pedro Gonzalez-Perez and Nicolas Pouyanne for discussions during the preparation of this work.\\

A detailed version of this work will be available soon.

\section{Jet schemes of toric surfaces} 
Let $\field{K}$ be an algebraically closed field. Let $X$ be a $\field{K}$-scheme of finite type over $\field{K}$ and let $m \in \mathbb{N}$. The functor $F_m :\field{K}-Schemes \longrightarrow Sets$
which to an affine scheme given by a $\field{K}-$algebra $A$ associates
 $$F_m(Spec~A)=Hom_\field{K}(Spec~A[t]/(t^{m+1}),X)$$
is representable by a $\field{K}-$scheme $X_m$ called the $m-$jet scheme of $X.$ \\For $m,p \in \mathbb{N}, m > p$, the truncation homomorphism $A[t]/(t^{m+1}) \longrightarrow A[t]/(t^{p+1})$ induces
a canonical projection $\pi_{m,p}: X_m \longrightarrow X_p.$ These morphisms clearly verify $\pi_{m,p}\circ \pi_{q,m}=\pi_{q,p}$
for $p<m<q$. This yields an inverse system whose limit $X_{\infty}$ is a scheme called the arc space of $X.$  Note that $X_0=X$. We denote the canonical projections $X_m\longrightarrow X_0$ by $\pi_{m}$ and $X_{\infty}\longrightarrow X_m$ by $\Psi_m$. See \cite{EM} for more about jet schemes.\\
Let $S$ be a singular affine toric surface defined over $\field{K}$ by the cone
$\sigma \subset N_{\IR}=\mathbb{R}^2$ generated by $(1,0)$ and $(p,q),$  where $0<p<q$ and $p,q$
are relatively prime. Let $(c_2,\ldots,c_{e-1})$ be the entries greater than or equal to two
occurring in the Hirzebruch-Jung continued fraction associated to
$q/p$. Then the embedding dimension of $S$ is $e$ (\cite{O}, section 1.6) . We suppose that $e>3,$ the case $e=3$ i.e. the rational double point $S=A_{c_2-1}$ is studied in \cite{Mo2}.
Analyzing the convex hull of
$\sigma^\vee \cap M$, where $M$ is the dual lattice of  $N$, Riemenschneider has exhibited the generators of the ideal  defining $S$ in $\mathbb{A}^e=\textnormal{Spec}\field{K}[x_1,\cdots,x_e]$ in \cite{R} ; these are:
$$E_{ij}=x_ix_j-x_{i+1}x_{i+1}^{c_{i+1}-2} x_{i+2}^{c_{i+2}-2}\cdots x_{j-2}^{c_{j-2}-2}x_{j-1}^{c_{j-1}-2}x_{j-1}, ~~\mbox{where} ~~1\leq i <j-1\leq e-1.$$
Let $f\in k[x_1,\ldots,x_e]$ ; for $m,p \in \mathbb{N}$ such that  $p\leq m,$ we set:
$$Cont^p(f)_{m}(resp.Cont^{>p}(f)_{m}):=\{\gamma \in S_m \mid ord_{\gamma }(f)=p(resp.>p)\},$$
$$Cont^p(f)=\{\gamma \in S_{\infty}\mid ord_{\gamma }(f)=p\},$$
where $ord_{\gamma }(f)$ is the $t-$order of $f\circ\gamma.$ 
\\For $a,b\in \mathbb{N},~b\not=0,$ we denote by $\lceil\frac{a}{b}\rceil$  the ceiling of $\frac{a}{b}$. For $i=2,\cdots, e-1,$
$ s \in \{1,\ldots,\lceil\frac{m}{2}\rceil\}$(i.e. $m\geq 2s-1\geq1)$ and
$ l \in \{s,\ldots,m_i^s\}, $ where $m_i^s:=min\{(c_i-1)s, (m+1)-s \},~$we set  $$D_{i,m}^{s,l}:=Cont^s(x_i)_{m} \cap Cont^{l}(x_{i+1})_{m}~~~~\mbox{and}~~~~C_{i,m}^{s,l}:=\overline{D_{i,m}^{s,l}}.$$

If $R$ is a ring, $I\subseteq R$ an ideal and $f \in R$, we denote by $V(I)$ the subvariety of $Spec~R$ defined by $I$
and by $D(f)$ the open set $D(f):=\textnormal{Spec}~R_f.$
\begin{lem} \label{eff}
 
For $i=2,\cdots,e-1,$ $s\geq 1,$    the ideal defining $C_{i,2s-1}^{s,s}$ in $\mathbb{A}^e_{2s-1}$ is 
$$I_{i,2s-1}^{s,s }=(x_j^{(b)},1 \leq j \leq e,0 \leq b < s).$$
Note that $C_{i,2s-1}^{s,s}$ does not depend on $i$. For $j=1,e,$ we set $C_{j,2s-1}^{s,s}:=C_{i,2s-1}^{s,s},~i=2,\cdots,e.$   
\end{lem}

\begin{dem}
Let's prove that  $D_{i,2s-1}^{s,s}=V(I_{i,2s-1}^{s,s})\cap D(x_i^{(s)}x_{i+1}^{(s)}).$
Let $\gamma \in \mathbb{A}^e_{2s-1}$ such that $ord_{\gamma}x_i=ord_{\gamma}x_{i+1}=s.$
So, we have $ord_{\gamma}x_i^{c_i}=c_is> 2s-1$ because $c_i\geq2.$ If moreover $\gamma$ lies in $S_{2s-1}$, then it satisfies $E_{i-1,i+1}$ mod $t^{2s},$ which is equivalent to $ord_{\gamma}x_{i-1} \geq s $, because $x_i^{c_i}\circ \gamma \equiv0$ mod $t^{2s}$
and $ord_{\gamma}x_{i+1}=s$. 
The same argument, using $E_{i-2,i},E_{i,i+2}$ and so on by induction, using the other $E_{ji}$'s and $E_{ij}$'s, gives that $ord_{\gamma}x_{j}\geq s.$
We deduce $$D_{i,2s-1}^{s,s}\subset V(I_{i,2s -1}^{s,s})\cap D(x_i^{(s)}x_{i+1}^{(s)}).$$
The opposite inclusion comes from the fact that a jet in $V(I_{i,2s-1}^{s,s})\cap D(x_i^{(s)}x_{i+1}^{(s)})
\subset \mathbb{A}^e_{2s}$
satisfies all the equations of $S$ modulo $t^{2s}.$ Since $V(I_{i,2s-1}^{s,s})\subset \mathbb{A}^e_{2s-1}$  is irreducible, the lemma follows.
\end{dem}
\begin{prop}\label{irr}
 
 For $i=2,\cdots,e-1,~m \in \mathbb{N},~  s \in \{1,\ldots,\lceil\frac{m}{2}\rceil\}$ and $l \in \{s,\ldots,m_i^s \}\}, $
$C_{i,m}^{s,l}$ is irreducible, and its codimension in  $\mathbb{A}^e_m$ is equal to $$ se+(m-(2s-1))(e-2).$$  
\end{prop} 
\begin{dem}
A similar argument to the one used in \ref{eff} shows that $\pi_{m,2s-1}(D_{i,m}^{s,l})\subset C^{s,s}_{i,2s-1}.$
Using Syzigies among $E_{jh},1\leq j <h-1\leq e-1,$ we prove that $$D_{i,m}^{s,l}=\{\gamma \in \mathbb{A}^e_m~;~ord_{\gamma}x_{i}=s,~ord_{\gamma}x_{i+1}=l,~ord_{\gamma}E_{j,h}\geq m+1~ \textnormal{for} ~(j,h)=(i-1,i+1),~ j=i,~h=i  \}.$$
This explicit description of $D_{i,m}^{s,l}$ shows that its coordinate ring is isomorphic to a polynomial ring over $Spec~\field{K}[x^{(s)}_i,x^{(l)}_{i+1}]_{x^{(s)}_ix^{(l)}_{i+1}},$ therefore its closure $C_{i,m}^{s,l}$ is irreducible. It also allows to compute its codimension. 

\end{dem}
\begin{remark}\label{nv} 
For $i=2,\cdots, e-1$ and $m,s\in \mathbb{N}$ such that $m\geq 2s-1$ and $l \in \{s,\ldots,m_i^s\},$ we have $\Psi^{-1}_m(D_{i,m}^{s,l})\not= \emptyset.$
\end{remark}
\begin{dem} Actually we prove that if $s\leq l \leq (c_i-1)s,$ then $Cont^s(x_i)\cap Cont^l(x_{i+1})\not=\emptyset.$
Let $u_i, i=1,\cdots ,e,$ be the system of minimal generators of $\sigma^\vee \cap M,$ i.e. $x^{u_i}=x_i.$ First note that since $(u_i,u_{i+1})$ is a $\mathbb{Z}-$basis of $M,$ there exists a unic $v\in N$ such that $<u_i,v>=s$ and $<u_{i+1},v>=l.$
It is enough to prove that $v\in \sigma.$ For $e=4,$ this is easy to check, and the lemma follows by induction on $e.$
 
\end{dem}

\begin{prop}\label{gs}
 Let $m,s \in \mathbb{N}$ such that $m\geq2s-1.$
\begin{enumerate}
\item  For $i=1,e,$ we have that $\pi_{m,2s-1}^{-1}(C_{i,2s-1}^{s,s}\cap D(x_{i}^{(s)}))$ is irreducible.
\item For $i=2,\cdots,e-1,~m\geq 2s-1,$ the irreducible components of $\overline{\pi_{m,2s-1}^{-1}(C_{i,2s-1}^{s,s}\cap D(x_{i}^{(s)}))}$ are the $C_{i,m}^{s,l}, l\in \{s,\cdots,m_i^s\}.$ 
\end{enumerate}
\end{prop}
 
\begin{dem}
We sketch the proof of (2), the proof of (1) is similar. We have already seen in the proof of proposition \ref{irr} that $D_{i,m}^{s,l}\subset \pi_{m,2s-1}^{-1}(C_{i,2s-1}^{s,s}\cap D(x_{i}^{(s)}))$ for $l\in \{s,\cdots,m_i^s\}.$
Using Syzigies among $E_{jh},1\leq j <h-1\leq e-1,$ we prove that $$\pi_{m,2s-1}^{-1}(C_{i,2s-1}^{s,s}\cap D(x_{i}^{(s)}))=\{\gamma \in \mathbb{A}^e_m~;~ord_{\gamma}x_{j}\geq s ~\textnormal{for}~ j=1,\cdots,e,~ord_{\gamma}x_{i}=s,$$
$$~ord_{\gamma}E_{j,h}\geq m+1~ \textnormal{for} ~(j,h)=(i-1,i+1),~ j=i,~h=i  \}.$$

This implies that the coordinate ring of the above set is isomorphic to a polynomial ring over the coordinate ring
of the locally closed subset of the $m-$jets of the $A_{c_i-1}$ singularity defined by $E_{i-1,i+1},$ consisting of those $\gamma$ such that $ord_{\gamma} x_i=s,~ord_{\gamma} x_{i-1}$ and $ord_{\gamma} x_{i+1}\geq s.$  The claim follows from the description of this latter.
 
\end{dem}

\begin{lem}
 $C_{i,m}^{s,s}=C_{i+1,m}^{s,m_{i+1}^s}$
 \end{lem}
\begin{dem}
This follows from the fact that an $m-$jet should verifie $(E_{i,i+2})$ modulo $m+1$, and from the explicit description in the proposition \ref{gs}. 
\end{dem}

Let $S_m^0:=\pi_{m}^{-1}(O)$, where $O$ is the singular point of $S.$ Note that $\overline{\pi_{m}^{-1}(S-\{0\})}$
is an irreducible component of $S_m$ of codimension $(m+1)(e-2)$ in $\mathbb{A}^e_m$ ; We will see that the irreducible components of $S_m^0$ have codimension less than or equal to $(m+1)(e-2),$ therefore they are irreducible components of $S_m.$
\begin{prop}\label{cov}
   $$S_m^0=\bigcup_{i\in \{2,...,e-1\},s \in \{1,\ldots,\lceil\frac{m}{2}\rceil\},l \in \{s,\ldots,m_i^s\}}  C_{i,m}^{s,l}.$$
\end{prop}
\begin{dem}
We first look at \textbf{the case m=2n+1}, $~n\geq 0.$ We claim that $$S_{2n+1}^0 =\bigcup_{i\in \{1,...,e\},s \in \{1,\ldots,n\}} \pi_{2n+1,2s-1}^{-1}(C_{i,2s-1}^{s,s}\cap D(x_{i}^{(s)}))\cup C_{i,2n+1}^{n+1,n+1}.~~~~~~~~(\diamond)$$
The proof of the claim is by induction on $n.$ 
By lemma \ref{eff}, we have that $S^0_1=C_{i,1}^{1,1}$ for any $i=1,...,e,$ hence the case $n=0.$
Using the inductive hypothesis for $n-1,$ and the fact that for  $s \in \{1,\ldots,n-1\}$ we have that $\pi_{2n-1,2s-1}\circ \pi_{2n+1,2n-1}=\pi_{2n+1,2s-1},$ we obtain:  
$$S_{2n+1}^0=\pi_{2n+1,2n-1}^{-1}(S_{2n-1}^0)=\bigcup_{i\in \{1,...,e\},s \in \{1,\ldots,n-1\}} \pi_{2n+1,2s-1}^{-1}(C_{i,2s-1}^{s,s}\cap D(x_{i}^{(s)}))\cup \pi_{2n+1,2n-1}^{-1}( C_{i,2n-1}^{n,n}).$$
The claim follows from the stratification 

$C_{i,2n-1}^{n,n}=\bigcup_{j=1,\cdots,e}(C_{i,2n-1}^{n,n}\cap D(x_{j}^{(n)}))\cup (C_{i,2n-1}^{n,n}\cap V(x_1^{(n)},
\cdots,x_e^{(n)})), $ \\
and from the fact that by lemma \ref{eff} $\pi_{2n+1,2n-1}^{-1}(C_{i,2n-1}^{n,n}\cap V(x_1^{(n)},
\cdots,x_e^{(n)})) =C_{i,2n+1}^{n+1,n+1}.$\\
We then conclude the proposition for $m=2n+1$ in two steps : First by using proposition \ref{gs} (2). Second, by deducing from that fact that the vector $(s,s) \in \sigma,$ hence $Cont^s(x_1)\cap Cont^s(x_2)\not=\emptyset,$ that $\pi_{2n+1,2s-1}^{-1}(C_{i,2s-1}^{s,s}\cap D(x_{2}^{(s)})) \cap  \pi_{2n+1,2s-1}^{-1} (C_{i,2s-1}^{s,s}\cap D(x_{1}^{(s)}))\not=\emptyset~;$ since by \ref{gs} (1) this latter is irreducible, its generic point coincides with the generic point of one of the irreducible compnents of $\overline{\pi_{2n+1,2s-1}^{-1}(C_{i,2s-1}^{s,s}\cap D(x_{1}^{(s)}))}.$  \\

\textbf{The case m =2(n+1)}, $n\geq0$ : by ($\diamond$) we just need to prove that  $$\pi_{2(n+1),2n+1}^{-1}(C_{i,2n+1}^{n+1,n+1})=
\cup_{\{i=2,\cdots,e-1~;~l=n+1,\cdots, m_i^{n+1}\}}C_{i,2(n+1)}^{n+1,l}.$$
 The proof is by induction on the embedding dimension. We show below the case $e=4$ :\\
If $c_2=c_3=2,$ then $m_i^{n+1}=n+1$ and by lemma \ref{eff}, $\pi_{2(n+1),2(n+1)-1}^{-1}(C_{i,2(n+1)-1}^{n+1,n+1})$  is defined in $\mathbb{A}^e_{2(n+1)}$
by $I_{i,2(n+1)-1}^{n+1,n+1}$ whose generators are coordinates and the ideal
$$(x_1^{(n+1)}x_3^{(n+1)}-x_2^{{(n+1)}^2},x_1^{(n+1)}x_4^{(n+1)}-x_2^{(n+1)}x_3^{(n+1)},x_2^{(n+1)}x_4^{(n+1)}-x_3^{{(n+1)}^2 }).$$
Therefore  $\pi_{2(n+1),2(n+1)-1}^{-1}(C_{i,2(n+1)-1}^{n+1,n+1})$ is irreducible (the above ideal is isomorphic to the ideal which defines the surface $S$) and is equal to $C_{j,2(n+1)}^{n+1,n+1}, j=2,3,$ since $D_{2,2(n+1)}^{n+1,n+1}=D_{3,2(n+1)}^{n+1,n+1}$ is dense in both. The subcases ($c_2=2$ and $c_3 \not=2)$  and ($c_2 \not=2$ and $c_3 \not=2)$ follow also easily.
\end{dem}  
\begin{thm} \label{th}
Let $m\in  \mathbb{N}, ~m\geq 1.$
Modulo the identifications 
$C_{i,m}^{s,s}=C_{i+1,m}^{s,m_{i+1}^s},$  the irreducible components of $S_m^0:=\pi_m^{-1}(0)$ are the $C_{i,m}^{s,l},~i=2, \cdots, e-1,$
$ s \in \{1,\ldots,\lceil\frac{m}{2}\rceil\} $ and
$ l \in \{s,\ldots,m_i^s \}\}.$ 
\end{thm} 
\begin{dem}
By proposition \ref{cov}, $S_m^{(0)}$ is covered by the $C_{i,m}^{s,l}.$ But apart from the identifications above, $C_{i,m}^{s,l}\not \subset C_{i',m}^{s,l'},$
because by proposition \ref{gs}, there exist hyperplane coodinates that contain the one but not the other, and by proposition \ref{irr} they have the same dimension. On the other hand $C_{i,m}^{s,l}\not \subset C_{i',m}^{s',l'},$ if $s <s',$ because by proposition \ref{gs} the $C_{i,m}^{s,l}$ has non-empty intersection with  $D(x_i^{(s)}),$ but $C_{i',m}^{s',l'}\subset V(x_i^{(s)})$. Finally, $C_{i',m}^{s',l'} \not \subset C_{i,m}^{s,l},$ because by proposition \ref{irr} the codimension of the first one is less than or equal to the codimension of the second one, and the theorem follows.
\end{dem}

\begin{rem} 
Given Theorem \ref{th}, remark \ref{nv} means that there are no vanishing components.
 \end{rem}
\begin{defi}
 Let $m \in \IN,~m\geq 1,$ and let $C$ be an irreducible component of $S_m^0.$ By Theorem \ref{th}, there exist  
$ s \in \{1,\ldots,\lceil\frac{m}{2}\rceil\}, $ 
$ l \in \{s,\ldots,m_i^s \} $ and  $i \in \{ 2,\cdots, e-1\}$ such that $C=C_{i,m}^{s,l}.$
We say that $C$ has index of speciality $s.$ Note that $s=ord_{\gamma}(M):=\textnormal{min}_{f\in M}\{ord_{\gamma}(f)\}$ where $M$ is the maximal ideal of the local ring $O_{S,0}$ and $\gamma$ the generic point of $C.$ 
\end{defi}
For $a,b\in \mathbb{N},~b\not=0,$ we denote by
$[\frac{a}{b}]$ the integral part of $\frac{a}{b}.$ 
 For $c,m \in \mathbb{N}$, let $m=q_cc+r$ be the euclidian division of $m$ by $c$. We set
$$N_c^{s}(m):=(sc-(2s-1)), for~~s=1,...,q_c~;~ N_c^{s}(m):=m-(2s-2),for~~s=q_c+1,...,\lceil\frac{m}{2}\rceil.$$
 For $m\in \mathbb{N},~m\geq 1$, we call $N(m)$ the number of irreducible component of $S_m^0.$ Then counting the irreducible components in the Theorem \ref{th} we find
\begin{cor}\label{nb}
 If all the $c_i$ are equal to $2,$ then $N(m)=\lceil\frac{m}{2}\rceil.$ Otherwise let $c_{i_1},...,c_{i_h}$ be the elements in $\{c_2,\ldots,c_{e-2}\}$ different from $2,$ then we have $N(m)=\sum_{s=1}^{\lceil\frac{m}{2}\rceil}(N^s_{c_{i_1}}(m)+(N^s_{c_{i_2}}(m)-1)+\ldots+(N^s_{c_{i_h}}(m)-1)).$
\end{cor}
\begin{cor} For $m\geq \mbox{max}\{c_i,~i=2,\cdots,e-1\},$ the number of irreducible components of $S_m^0,$ with index of speciality 
 $s=1,$ is equal to the number of exceptional divisors that appear on the minimal resolution of $S.$
\end{cor}
\begin{proof}
This comes from the comparaison of corollary \ref{nb} with corollary $1.23$ in \cite{O} page 29. 
 \end{proof}

\end{document}